\documentclass[11pt]{amsart}
\usepackage[utf8]{inputenc} % allows spec chars in bib file
\usepackage{mathtools,amssymb,amsxtra,mathrsfs}
\usepackage{bm} % for bold math
\usepackage{xcolor,enumitem}

\usepackage{indentfirst}
 
\newcommand{\tss}{\textsuperscript}

\usepackage[pagewise,mathlines]{lineno}
\nolinenumbers
\newcommand{\w}{\omega}

\newcommand{\mc}[1]{\mathcal{#1}}

\newcommand{\set}[1]{\{#1\}}

\newcommand{\st}{\,\mathrm{:}\,}

\newcommand{\Qc}{\mathbb{P}}
\newcommand{\Q}{\mathbb{Q}}
\newcommand{\R}{\mathbb{R}}

\let\phi\varphi

\newcommand{\scpt}{$\sigma$-compact}
\newcommand{\coa}{co-analytic}
\newcommand{\proj}{projective}
\newcommand{\gdelta}{$G_\delta$}
\newcommand{\VL}{$\bm{V = L}$}

\renewcommand{\bar}{\overline}

\usepackage{graphicx} % better backslash
\newcommand\rsetminus{\mathbin{\mathpalette\rsetminusaux\relax}}
\newcommand\rsetminusaux[2]{\mspace{-3mu}\raisebox{\rsmraise{#1}\depth}{\rotatebox[origin=c]{-35}{$#1\smallsetminus$}}\mspace{-3mu}}
\newcommand\rsmraise[1]{%
  \ifx#1\displaystyle 1 \else
    \ifx#1\textstyle .3 \else
      \ifx#1\scriptstyle .2 \else
        .1%
      \fi
    \fi
  \fi}

\let\setminus\bigsetminus

\theoremstyle{plain}      % bold + italics

\theoremstyle{definition} % bold + upright
\newtheorem{thrm}{Theorem}[section]
\newtheorem*{thrm*}{Theorem}
\newtheorem{cor}[thrm]{Corollary}
\newtheorem{lemma}[thrm]{Lemma}
\newtheorem*{lemma*}{Lemma}
\newtheorem{propn}[thrm]{Proposition}
\newtheorem*{propn*}{Proposition}
\newtheorem{defn}{Definition}[section]
\newtheorem*{defn*}{Definition}
\newtheorem*{rmk*}{Remark}
\newtheorem{probm}{Problem}

\theoremstyle{remark}     % italics + upright

\newtheoremstyle{cited}% from amsthmdoc.pdf
  {3pt}% (space above)
  {3pt}% (space below)
  {}% (body font)
  {}% (indent amount)
  {\bfseries}% (theorem head font)
  {.}% (punctuation after theorem head)
  {.5em}% (space after theorem head)
  {\thmname{#1} \thmnumber{#2} \thmnote{\normalfont#3}}% (theorem head spec (can be left empty, meaning 'normal'))

\theoremstyle{cited}

\newtheorem{cthrm}[thrm]{Theorem}
\newtheorem{ccor}[thrm]{Corollary}
\newtheorem{cpropn}[thrm]{Proposition}
\newtheorem{clemma}[thrm]{Lemma}

\begin{document} 
\title{Completely Baire spaces, Menger spaces, and projective sets}
\author{Franklin D. Tall{$^1$} \and Lyubomyr Zdomskyy{$^2$}}
\footnotetext[1]{Research supported by NSERC Grant A-7354.}
\footnotetext[2]{The second author thanks the Austrian Science Fund FWF (Grant I 2374 N35) for generous support for this research.}

\date{\today}

\keywords{Menger, \scpt, Baire, completely Baire, Polish, analytic, co-analytic, projective, Open Graph Axiom}%, Hurewicz Dichotomy}

\subjclass[2010]{54A35, 03E35, 54D45, 03F15, 54F52, 54E50, 54G20, 03C95, 03E15, 03C95, 03E60, 54G20}

\begin{abstract} 
W.~Hurewicz proved that analytic Menger sets of reals are \scpt{} and that co-analytic completely Baire sets of reals are completely metrizable. 
It is natural to try to generalize these theorems to projective sets. 
This has previously been accomplished by $\bm{V = L}$ for projective counterexamples, 
and the {\bf Axiom of Projective Determinacy} for positive results.
For the first problem, the first author, S.~Todorcevic, and S.~Tokg\"oz have produced a finer analysis with much weaker axioms.
We produce a similar analysis for the second problem, showing the two problems are essentially equivalent. 
We also construct in ZFC a separable metrizable space with $\w$\tss{th} power completely Baire, yet lacking a dense completely metrizable subspace.
This answers a question of Eagle and Tall in {\it Abstract Model Theory}.
\end{abstract}
\maketitle

\medskip

\section{Introduction}
It is a common theme in {\it Descriptive Set Theory} that statements about simply definable sets are true, 
e.g.~``All Borel sets are Lebesgue measurable,'' 
but that the Axiom of Choice entails the existence of a non-constructive counterexample:
``There is a non-measurable set of reals.''
For subsets of $\R$ that are definable but not so simply, the situation is more complex; 
e.g.~$\bm{V = L}$ implies there is a continuous image of the complement of a continuous image of a Borel set that is not measurable, 
but the {\bf Axiom of Projective Determinacy} ({\bf PD}) implies that the class of subsets of $\R$ obtained by closing the class of Borel sets under complement and continuous image contains only measurable sets. 
Among other places, these phenomena have been investigated with regard to two classical theorems of W.~Hurewicz. 
We first give the relevant definitions.

\begin{defn}
A topological space is {\bf Menger} if whenever $\set{\mc{U}_n}_{n < \w}$ is a sequence of open covers, there exists $\set{\mc{V}_n}_{n < \w}$, $\mc{V}_n \subseteq \mc{U}_n$, $\mc{V}_n$ finite, such that $\set{\bigcup \mc{V}_n \st n < \w}$ is a cover.
\end{defn}

\begin{defn}
A subset $A$ of $\R$ is {\bf analytic} if it is a continuous image of a Borel set. 
$C \subseteq \R$ is {\bf \coa} if $\R \setminus A$ is analytic.
$P \subseteq \R$ is {\bf \proj} if it is in the class of subsets of $\R$ obtained by closing the Borel sets under complementation and continuous real-valued image.
\end{defn}

\begin{defn}
A topological space is {\bf Baire} if the intersection of any countable family of dense open sets is dense. 
A space is {\bf completely Baire} if each closed subspace is Baire.
\end{defn}

\begin{defn}
A topological space is {\bf Polish} if it is separable and completely metrizable.
\end{defn}

\begin{cpropn}[\cite{H1}]
Analytic Menger subsets of $\R$ are \scpt.
\end{cpropn}

\begin{cpropn}[\cite{H2}]
Co-analytic completely Baire subsets of $\R$ are Polish.
\end{cpropn}

\begin{probm}
Are ``definable'' Menger subsets of $\R$ \scpt?
\end{probm}

\begin{probm}
Are ``definable'' completely Baire subsets of $\R$ Polish?
\end{probm} 

\section{Results from the literature}

We refer to \cite{Kec95} for descriptive set theory, and to \cite{Kan} for \VL{} and large cardinals.
Problem 1 was investigated by Miller and Fremlin \cite{MF} in 1988. 
They proved that:

\begin{propn}
\VL{} implies there is a co-analytic Menger set of reals that is not \scpt.
\end{propn}

\begin{propn}
{\bf PD} implies every projective Menger set of reals is \scpt.
\end{propn}

{\bf PD} is regarded as ``true'' by many descriptive set theorists but has quite large cardinal consistency strength. 
\cite{MF} was extended by the first author and S.~Tokg\"oz to consider spaces that were not necessarily metrizable \cite{TT}.
They also noted that: 

\begin{propn}
The {\bf Axiom of Co-analytic Determinacy} ($\bm{\Pi_1^1}$-Determinacy) implies co-analytic Menger sets of reals are \scpt.
\end{propn}

\begin{cor}
If there is a measurable cardinal, then co-analytic Menger sets of reals are \scpt.
\end{cor}

Problem 2 is solved from the appropriate determinacy assumptions in \cite[28.20]{Kec95}.
Medini and Zdomskyy in 2015 proved 

\begin{cthrm}[\cite{MZ}]
\VL{} implies there is an analytic, completely Baire set of reals which is not Polish.
\end{cthrm}

The hypotheses concerning Problem 1 were considerably sharpened in \cite{TTT}:

\begin{cthrm}[\cite{TTT}] $\omega_1^{L[a]} < \omega_1$, for all reals $a$, if and only if every co-analytic Menger set of reals is \scpt.
\end{cthrm}

\begin{ccor}[\cite{TTT}]
The assertion that every co-analytic Menger set of reals is \scpt{} is equiconsistent with the existence of an inaccessible cardinal.
\end{ccor}

We shall prove analogous results with respect to Problem 2:

\begin{thrm}
$\w_1^{L[a]} < \w_1$, for all reals $a$, if and only if every analytic, completely Baire set of reals is Polish.
\end{thrm}

\begin{cor}
The assertion that every analytic completely Baire set of reals is Polish is equiconsistent with the existence of an inaccessible cardinal.
\end{cor}

\section{The Menger property and the completely Baire property}

There is a surprising connection between {\bf Problems 1 and 2}:

\begin{thrm}
Suppose $X$ is a Menger set of reals. Then $\R \setminus X$ is completely Baire. % $X$ is \scpt{} if and only if $\R \setminus X$ is Polish. $X$ is \coa{} if and only if $\R \setminus X$ is analytic. $X$ is \proj{} if and only if $\R \setminus X$ is \proj.
\end{thrm}
\begin{proof}
% The last three assertions are obvious. For the first, 
% let us first consider the case when $X$ is a Menger subset of a Cantor set $K$. 
% We shall prove $K \setminus X$ is completely Baire. 
% It follows that $\R \setminus X$ is completely Baire.
We need: 
\begin{cpropn}[\cite{H2} (For a proof in English, see \cite{vM})]
A metrizable space is completely Baire if and only if it does not include a closed copy of the space $\Q$ of rationals.
\end{cpropn}

Now suppose $\R \setminus X$ is not completely Baire. Then there is a copy $Q$ of $\Q$ closed in $\R \setminus X$. 
% Then $\bar{Q} \cap X$ is homeomorphic to the space $\Qc$ of irrationals. But $\bar{Q} \cap X$ is closed in $X$ so should be Menger, contradiction.
%
% Now consider the case when $X$ is any Menger set of reals. It is folklore (see e.g.~\cite{Kec95}) that every separable metrizable space is a perfect image of a 0-dimensional one, and hence of a subspace of $2^\w$.
%
% Let $\pi$ perfectly map $S \subseteq 2^\w$ onto $\R$. 
% Without loss of generality, $S$ is dense in $2^\w$.
% Then $\pi^{-1}(X)$ is Menger, so $2^\w \setminus \pi^{-1}(X)$ is completely Baire. $\pi^{-1}(\R \setminus X = S \setminus \pi^{-1}(X)$. $S$ is \v{C}ech-complete, so is \gdelta{} in $2^\w$. Then $S \setminus \pi^{-1}(X)$ is a \gdelta{} in $2^\w \setminus \pi^{-1}(X)$, so is completely Baire \cite{MZ}. Claim: 
%
% \begin{lemma}
% The perfect image of a completely Baire separable metrizable space is completely Baire. 
% \end{lemma}
%
% \begin{proof}
% Let $A$ be completely Baire and metrizable. Let $p$ map $A$ perfectly onto $B$. 
% Then $B$ is separable metrizable. 
% Suppose on the contrary that $Q$ is a closed copy of $\Q$ in $B$. 
% Then $p^{-1}(Q)$ is closed and \scpt{} in $A$. 
% Then $p^{-1}(X)$ is \coa{} and completely Baire, so is Polish. 
% But then $Q$ is Polish, which is absurd. 
% \end{proof}
Then $\bar{Q}\setminus Q$ is Polish, nowhere locally compact and zero-dimensional (since it does not include any interval). But then it is homeomorphic to the space $\Qc$ of irrationals, which is not Menger, despite being a closed subspace of $X$. This is a contradiction.
Thus $\R \setminus X$ is completely Baire, which was to be proved.
\end{proof}

\section{An Open Graph Axiom}

In \cite{To}, 
Todorcevic introduced what he called the {\bf Open Coloring Axiom}. 
Unfortunately that name had earlier been used by Abraham--Rubin--Shelah \cite{ARS}.
This has caused some confusion, so Todorcevic has renamed his axiom the {\bf Open Graph Axiom}. 
A variation of this axiom was introduced in \cite{F}. 
Slightly modifying Feng's notation, 
we have: 

\begin{defn}
Let $\Gamma$ be a collection of subsets of $\R$.
OGA*($\Gamma$): 
Let $X$ be any member of $\Gamma$. 
Let $[X]^2 = K_1 \cup K_2$ be a partition with $K_1$ open in the topology on $[X]^2$ inherited from $X^2$. 
Either there is a perfect $A \subseteq X$ with $[A]^2 \subseteq K_1$ or $X = \bigcup_{n < \w} A_n$ with $[A_n]^2 \subseteq K_2$ for all $n < \w$.
\end{defn}

Replacing ``perfect'' with ``uncountable'' one gets the {\bf Open Graph Axiom}. % (formerly Todorcevic's {\bf Open Coloring Axiom}). 
The {\bf Open Graph Axiom} implies the continuum hypothesis fails \cite{To}, but {\bf OGA*(projective)} holds in the model obtained by collapsing an inaccessible to $\w_1$ by finite conditions \cite{F}.
CH can be arranged to hold in such a model, so {\bf OGA*(projective)} surprisingly does not imply {\bf OGA}.

\begin{thrm}
The following are equivalent. 
\begin{enumerate}[label=\alph*)]
\item {\bf OGA*(\coa)};
\item \coa{} Menger subsets of $\R$ are \scpt;
\item analytic completely Baire subsets of $\R$ are Polish;
\item for every $a \in \R$, $\w_1^{L[a]} < \w_1$;
\item every uncountable \coa{} set includes a perfect set.
\end{enumerate}
\end{thrm}

\begin{proof}
The equivalence of a) and b) is in \cite{TTT}; of a), d), and e) in \cite{F}. 
The implication from c) to b) follows from Theorem 3.1. 
To obtain c) from a), we recall that in \cite{TTT} it is shown that: 

\begin{propn}
If $\Gamma$ is closed under continuous pre-images, then 
OGA*($\Gamma$) implies that if $A \in \Gamma$ is not \scpt, then there is a compact $K \subseteq \R$ such that $K \cap A$ is homeomorphic to $\Qc$, the space of irrationals, while $K \cap (\R \setminus A)$ is homeomorphic to $\Q$.
\end{propn}

Now suppose $B$ is completely Baire, analytic, and not Polish. 
Then $\R \setminus B$ is co-analytic and not \scpt, so there is a compact $K$ with $K \cap B$ homeomorphic to $\Q$. 
But $K \cap B$ is closed in $B$, contradicting $B$ being completely Baire.
\end{proof}

The next result extends Theorems 2.7 and 4.1 to projective sets.

\begin{thrm}
The following are equiconsistent:
\begin{enumerate}[label=\alph*), leftmargin=4em]
\item {\bf OGA*(projective)};
\item projective Menger subsets of $\R$ are \scpt;
\item projective completely Baire subsets of $\R$ are Polish;
\item there is an inaccessible cardinal;
\item every uncountable projective set includes a perfect set.
\end{enumerate}
\end{thrm}
\begin{proof}
Recall that the complement of a projective set is projective
and that the continuous pre-image of a projective set is projective \cite{Kec95}. 
Also recall that if $\w_1^{L[a]} < \w_1$, then $\w_1$ is inaccessible in $L$, so it is consistent there is an inaccessible cardinal. 
% Thus, from \cite{TTT}, we have $\neg\text{d)}$ implies $\neg\text{b)}$.
The implication from a) to b) is in \cite{TTT} and is clear from 4.2. 
That c) implies b) is by 3.1. The consistency of b) implies the consistency of d) by 4.1.
% for projective sets, as is the equivalence of b) and c). 
Feng \cite{F} proved the equiconsistency of a) and d), 
and that a) is equivalent to e). 
\end{proof}

\section{An Application to Model Theory}

\begin{defn}
A {\it theory} is a set of sentences closed under logical consequence. A {\it type} is a collection of formulas. 
An {\it $n$-type} is a collection of formulas, each with exactly $n$ free variables.
% The set of types can be made into a topological space by taking the sets of types containing a particular formula as basic open sets. 
An $n$-type is called {\it isolated} 
for a theory $T$ if there is a formula $\phi$ with exactly $n$ free variables 
such that in every model $\mc{M}$ of $T$, 
any $n$-tuple satisfying $\phi$ must satisfy all the elements of the type.
The {\it Omitting Types Theorem} ({\bf OTT}) (for first-order logic) asserts:

\smallskip
% Let $T$ be a countable theory. 
% Let $\Sigma$ be a type. 
% Provided $(T \cup \set{\phi}) \not\models \Sigma$, 
% for a single formula $\phi$, 
% there is a model $\mc{M}$ of $T$ such that $\mc{M} \models \neg \sigma$, for all $\sigma \in \Sigma$.
Let $\{\Sigma_k\}_{k<\w}$ be a countable collection of non-isolated types.
Then there is a model $\mc{M}$ of $T$
such that for each $k$ and each $\sigma \in \Sigma_k$, $\mc{M} \models \neg \sigma$.
\end{defn}

\noindent The {\bf OTT} follows from the {\bf Compactness Theorem} for first-order logic, but also holds for certain other logics for which compactness fails, and is a useful substitute for compactness.
The {\bf OTT} is often proved by a Baire category argument; the exact relationship between the {\bf OTT} and the {\bf Baire Category Theorem} is investigated in Eagle and Tall \cite{ET}.
Just as the existence of a winning strategy for Non-empty in the Banach-Mazur game ({\it weak $\alpha$-favorability}) is strictly stronger than just being a Baire space, one can formulate a game version of the {\bf OTT} and ask whether it too is strictly stronger. 
Eagle and Tall show how to define abstract logics from arbitrary topological spaces and prove: 

\begin{thrm}
There is an abstract logic satisfying the {\bf OTT} but not its game version if there exists a separable metrizable $Y$ such that $Y^\omega$ is completely Baire but does not include a dense completely metrizable subspace.
\end{thrm}

% Such a subspace is then constructed 
They then % construct such a space from
note that a {\it non-meager $P$-filter} is such a $Y$, but such a filter is not known to exist in ZFC, although its existence follows from a variety of consistent set-theoretic hypotheses.

\begin{thrm}
Such a space and hence such a logic exist in ZFC.
\end{thrm}

We shall exploit the work of Repov\v{s}, Zdomskyy, and Zhang, 2014.

\begin{cpropn}[\cite{RZZ}]
There is a subspace $X$ of $2^\omega$ such that $X$ is Menger, non-meager, and is a filter extending the Fr\'echet filter. 
\end{cpropn}

{\bf $\bm {Y = 2^\omega\setminus X}$ is the desired space!} 

\smallskip 

First we prove: 
\begin{lemma}
If $X \subseteq 2^\omega$ is Menger, then $(2^\omega \setminus X)^\omega$ is completely Baire.
\end{lemma}

\begin{proof}
It suffices to prove that $(2^\omega)^\omega \setminus (2^\omega\setminus X)^\omega$ is Menger. 
$(2^\omega)^\omega \setminus (2^\omega \setminus X)^\omega = \bigcup_{n < \omega} Y_n$, where 
\begin{gather*}
Y_n = \prod_{k < \omega} Y_{n,k}, \\
Y_{n,k} = \begin{cases}X & \text{for }n = k \\ 2^\omega & \text{for } n \ne k\end{cases}
\end{gather*}
Each $Y_n$ is Menger, so $(2^\omega)^\omega \setminus (2^\omega\setminus X)^\omega$ is Menger, because a countable union of Menger spaces is Menger.
\end{proof}

To prove $Y^\omega = (2^\omega \setminus X)^\omega$ has no completely metrizable dense subspace, suppose on the contrary that it has such a subspace $P$. 
Let $Z = \pi_1(P)$ be the projection onto the first coordinate of the product. 
Then $Z$ is dense in $Y$.
Since $P$ is \gdelta{} in the separable metrizable space $Y^\omega$, $Z$ is analytic and therefore has the {\bf Baire Property}, i.e.~$Z = U \triangle N$, where $U$ is open and $N$ is meager. 
We divide into cases, depending on whether or not $Z$ is meager. 
Either case leads to a contradiction.

\smallskip
\noindent {\it Case 1}: $Z$ is meager in $Y$.

Then, considering $Y^\omega$ as $Y\times Y^\omega$, we have $Z\times Y^\omega$ is dense in $Y^\omega$. But since $Z$ is meager in $Y$, $Z \times Y^\omega$ is meager in $Y^\omega$. But $P \subseteq Z \times Y^\omega$, and a meager set can't include a dense Polish subspace.

\smallskip
\noindent {\it Case 2}: 
Then since $Z$ has the Baire Property there is a non-empty open $U \subseteq Y$ and a dense $G_\delta$ T in $U$ such that $T \subseteq Z$ \cite[8.26]{Kec95}.
Without loss of generality, let $U$ be $Y \cap [s]$, where $[s]$ is basic open in $2^\omega$, 
$[s] = \set{f \in 2^\omega \text{ : } f|n = s}$, where $s \in 2^n$.
Let us enumerate $2^n = \set{s_i \text{ : } i \le 2^n}$. 
Without loss of generality, $s = s_0$. For $i \ne 0$, let $\varphi_i:[s_0]\to[s_i]$ be defined by 
\[\varphi_i(y) = s_i^n(y\!\upharpoonright\![n,+\infty)).\]
Then observe that $\varphi_i[[s_0 \cap X]] = [s_i]\cap X$, because $X$ is a filter extending the Fr\'echet filter, and so is closed under finite modifications. 
Let $T_0 = T$. For $0 < i \le 2^n$, let $T_i = \varphi_i[T]$. Then $T_i$ is a dense $G_\delta$ in $[s_i]$, since $\varphi_i$ is a homeomorphism.
Moreover, $T_i \cap X = \varnothing$. 
Then $R = \bigcup_{i \le 2^n} T_i$ is a \gdelta{} in $2^\omega$ and is dense there. But $R \cap X = \varnothing$, contradicting $X$ being non-meager. \qed

\begin{rmk*}
The complement of a completely Baire subset of $\R$ need not be Menger. 
A Bernstein set $B$ and its complement are both completely Baire.
We shall show that $B$ has a non-Menger closed subspace and so is not Menger.
Let $\mathbb{K}$ be the Cantor set. 
Since both $B$ and $\mathbb{K}\setminus B$ are Bernstein in $\mathbb{K}$, 
$\mathbb{K}\setminus B$ is dense in $\mathbb{K}$.
Fix a countable dense set $Q \subseteq \mathbb{K}\setminus B$. 
Then $\mathbb{K}\setminus Q$ is a copy of the space of irrationals.
$B \cap \mathbb{K} = B \cap (\mathbb{K}\setminus Q)$ is closed in $B$ 
and is a Bernstein subset of $\mathbb{K}\setminus Q$.
But $\mathbb{K}\setminus Q$ is homeomorphic to $^\omega\omega$.
It remains to quote:
\end{rmk*}

\begin{clemma}[\cite{BTZ}]
Bernstein subsets of $^\omega\omega$ are not Menger. \qed
\end{clemma} 

\bibliographystyle{alpha}

\medskip

{\small 
Franklin D. Tall, Department of Mathematics, University of Toronto, Toronto, Ontario, M5S 2E4, CANADA

{\it e-mail address:} {\tt f.tall@utoronto.ca}

Lyubomyr Zdomskyy, Institut für Diskrete Mathematik und Geometrie, Technische Universität Wien; Wiedner Hauptstra{\ss}e 8-10/104, 1040 Wien, AUSTRIA

{\it e-mail address:} {\tt lzdomsky@gmail.com}
}

\end{document}